\documentclass[10pt,a4letter,reqno]{amsart}
\usepackage{amscd,amsmath,amssymb,amsxtra,latexsym,amsthm}
\usepackage[english]{babel}
\usepackage[utf8]{inputenc}
\usepackage[all]{xy}

\newtheorem{theo}{Theorem}[section]
\newtheorem{prop}[theo]{Proposition}
\newtheorem{coro}[theo]{Corollary}
\newtheorem{lem}[theo]{Lemma}
\newtheorem{defi}[theo]{Definition}

\newtheorem{re}[theo]{Remark}

\newtheorem{cla}[theo]{Claim}

\newcommand{\R}{\mathbb{R}}

\newcommand{\G}{\mathbb{G}}

\newcommand{\dis}{\displaystyle}
\newcommand{\dist}{\textrm{dist}}

\title[Horizontal gradient of polynomial functions]{Horizontal gradient of polynomial functions for the standard Engel structure on $\R^4$}         % 

\author{Si Tiep DINH$^\dagger$}
\address{$^\dagger$Institute of Mathematics, 18 Hoang Quoc Viet Road, Cau Giay District, 10307, Hanoi, Vietnam}
\email{$^\dagger$dstiep@math.ac.vn}

\author{Krzysztof KURDYKA$^\ddagger$}
\address{$^\ddagger$Laboratoire de Math\'ematiques (LAMA) UMR-5127 CNRS, Campus Scientifique, 73376 Le Bourget-du-Lac cedex, France}
\email{$^\ddagger$Krzysztof.Kurdyka@univ-savoie.fr}

\thanks{This work was supported by the Fields Institute, ANR grant STAAVF (France), LIA Formath Vietnam Project and Vietnam National Foundation for Science and Technology Development (NAFOSTED) grant 101.01-2011.44.}



\begin{document}
\maketitle

\begin{abstract}
We investigate the set $V_f$ of horizontal critical points of a polynomial function $f$ for the standard Engel structure defined by the 1-forms $\omega_3=dx_3-x_1dx_2,$ $\omega_4=dx_4-x_3dx_2$, endowed with the sub-Riemannian metric $g_{SR}=dx_1^2+dx_2^2$. For a generic polynomial, we show that the intersection of any fiber of $f$ and $V_f$ does not contain a horizontal curve. Then we prove that each trajectory of the horizontal gradient of $f$ approaching the set $V_f$ has a limit. 
\end{abstract}

\section{Introduction}

An {\bf Engel structure} $\Delta$ is a non-integrable distribution of rank $2$ on a $4$-dimensional manifold which satisfies the following conditions:
\begin{equation}
\begin{array}{lcccl}
$$rank[\Delta,\Delta]=3,$$ \\
$$rank[\Delta,[\Delta,\Delta]]=4,$$ \\
\end{array}
\end{equation}
where $[.,.]$ denotes the Lie bracket. Engel structures are stable (or generic) in the sense that all $C^2$-small perturbation of an Engel structure is still an Engel structure. A manifold with an Engel structure is called an {\bf Engel manifold}. In this paper, we will deal with the standard Engel structure on $\R^4$, defined by the 1-forms  $\omega_3=dx_3-x_1dx_2$ and $\omega_4=dx_4-x_3dx_2$, which is generated by the following vector fields:
\begin{equation}\label{dis}(\Delta)
\left \{\begin{array}{lcccl}
$$X_1 =\displaystyle\frac{\partial}{\partial x_1}$$ \\
$$X_2 =\displaystyle\frac{\partial}{\partial x_2}+x_1\frac{\partial}{\partial x_3}+x_3\frac{\partial}{\partial x_4}$$
\end{array}\right.\end{equation}
We denote by $\Delta$ this Engel structure. By a result of Engel \cite{Eng}, at each point of an Engel manifold, there exists a neighborhood with local coordinates $x_1, x_2, x_3, x_4$ such that the Engel structure is generated by the vector fields $X_1,X_2$ above. So locally, all Engel structures are isomorphic.

Let us fix $g_{SR}=dx_1^2+dx_2^2$, a {\bf sub-Riemannian metric} on $\Delta$ for which the system $\{X_1,X_2\}$ is orthonormalized. Let $\displaystyle X_3=\frac{\partial}{\partial x_3},X_4=\frac{\partial}{\partial x_4}$. Then $g=dx_1^2+dx_2^2+\omega_3^2+\omega_4^2$ is the Riemannian metric on $\R^4$ extending $g_{SR}$ and making the system $\{X_1,X_2,X_3,X_4\}$ orthonormalized. The pair $(\Delta,g_{SR})$ is called a {\bf sub-Riemannian structure} on $\R^4$.

We refer to \cite{AS, Bel,Cho,Ger,Gro,Jea,Stri,Sus,Vog,Zhi} for further informations about subriemannian geometry and Engel structures.

Let $\R_d[x]:=\{f\in\R[x],\deg f\leq d\}$, the set $\R_d[x]$ is furnished a structure of finite dimensional vector space. We may identify a polynomial $f\in\R_d[x]$ with its coefficient vector and identify $\R_d[x]$ with $\R^{dim\R_d[x]}$. So $\R_d[x]$ is endowed with the usual topology of $\R^{\dim\R_d[x]}$. For each $f\in\R_d[x]$, we can associate a vector field $\nabla^hf $, called {\bf horizontal gradient} of $f$, which is the projection of the Riemannian gradient $\nabla f$ on the distribution $\Delta$ with respect to the metric $g_{SR}$. In the coordinate system $\{X_1,X_2,X_3,X_4\}$, 
$$\displaystyle\nabla f=\sum_{i=1}^4(X_if)X_i.$$
By definition 
\begin{equation}\label{gh}\displaystyle\nabla^h f=(X_1f)X_1+(X_2f)X_2.
\end{equation}
For convenience,  we will sometimes identify $\nabla^hf$ with the map $x\mapsto (X_1f(x),X_2f(x))$. The notion of horizontal gradient appeared in some recent works \cite{BHT,Mag} on the Carnot groups. Denote
\begin{equation}
V_f=\{\nabla^h f=0\}
\end{equation}
{\bf the set of horizontal critical points} of $f$. Note that if $x\in V_f$, then either $x$ is a riemannian critical point of $f$, i.e., $\nabla f(x)=0$, or $\Delta_x\subset T_xf^{-1}(f(x))$.

\begin{defi} An almost everywhere differentiable curve $\gamma:(t_1,t_2)\rightarrow\R^4$ is called a horizontal curve if its derivative $\dot\gamma(t)$ is almost everywhere contained in $\Delta_{\gamma(t)}$.
\end{defi}

Let us firstly recall the \L ojasiewicz inequality \cite{BR,Cos,Loj2,Loj3,Loj} for Riemannian gradient $\nabla f$. This inequality is an important tool to study the trajectories of $\nabla^hf$. Assume that $f$ is an analytic function and $\nabla f(x_0)=0$, then the \L ojasiewicz inequality says that in a neighborhood $U_0$ of $x_0$,
\begin{equation}\label{LI} ||\nabla f(x)|| \geq C|f(x) - f(x_0)|^\alpha\end{equation}
for some $0<\alpha<1$ and $C>0$. As a consequence, in \cite{Loj3,Loj4}, \L ojasiewicz proved that locally the length of each trajectory $x(t)$ of $\nabla f$ is bounded uniformly, hence the limit $\displaystyle\lim_{t\rightarrow\infty}x(t)$ exists, i.e., if $x(t_m)\rightarrow x_0$ for some sequence $t_m\rightarrow\infty$, then $x(t)\rightarrow x_0$ as $t\rightarrow\infty$.

In \cite{DKO}, we have consider the trajectories of the horizontal gradient for a class of codimension one distribution, called splitting distribution, on $\R^n$. This class contains, in particular, contact structures. We have noticed that the \L ojasiewicz inequality does not hold for horizontal gradient and some new phenomena appear, for example, a trajectory of the horizontal gradient may have unbounded length or accumulate to a cycle. This presents a major obstruction to the study of the horizontal gradient's trajectories since the techniques of \L ojasiewicz do not apply to the horizontal gradient. It turns out that to overcome this difficulty, we have to study the restriction of the function on its set of horizontal critical points $f|_{V_f}$. The conclusion is that for a generic polynomial, the restriction $f|_{V_f}$ is a Morse function and the behavior of the trajectories of horizontal gradient is similar to those of Riemannian gradient.

Dealing with the standard Engel structure $\Delta$ above, in general, the \L ojasiewicz inequality does not hold for horizontal gradient. In fact, if $V_f\not=\emptyset$ and $f|_{V_f}$ is a non constant  function, then we can not have the \L ojasiewicz inequality. Now studying $f|_{V_f}$ is not sufficient to understand the behavior of the trajectories of $\nabla^hf$, and it turns out that the horizontal curves contained in $V_f$ play an important role in the study of the trajectories of $\nabla^hf$.

The paper is organized as follow. In the second section, we recall some basic notions and results of semi-algebraic geometry. In the third part, we study the generic properties of the dimension of $V_f$. On the generic dimension of $V_f$, we obtain the following.

\medskip\noindent\textbf{Theorem \ref{dimension}.}\textit{ Generically, $V_f$ is smooth and of dimension $2$ or is empty, i.e., the set
$$K_d=\{f\in\R_d[x]:V_f \text{ is smooth of dimension } 2 \text{ or is empty}\}$$
contains a semi-algebraic open dense set in $\R_d[x]$.}

The proof of this result is based on the transversality theorem with parameters. It's quite similar to the proof of Theorem 4.4 in \cite{DKO}.

The following result plays an important role in the study of the trajectories of horizontal gradient.

\medskip\noindent\textbf{Theorem \ref{noncontient}.}\textit{ Generically, the set $V_f\cap f^{-1}(t)$ does not contain a horizontal curve, i.e., the set
$$M_d=\{f\in\R_d[x]: V_f\cap f^{-1}(t) \text{ does not contain a horizontal curve for every }t\}$$
contains a semi-algebraic open dense set in $\R_d[x]$.}

In the last part, we show that

\medskip\noindent\textbf{Theorem \ref{havelimit}.}\textit{ Generically, the trajectories of $\nabla^hf$ in a compact set $B$ have a limit, i.e., the set
$$N_d=\{f\in\R_d[x]: \text{ the trajectories of }\nabla^hf \text{ have a limit on } \partial B\cup V_f\}$$
contains a semi-algebraic open dense set in $\R_d[x]$.}

%------------------------------------------------------------------------------------------------------------------------------

\section{Briefs on transversality and semi-algebraicity}

\hskip 0.5cm We recall some principal definitions and results of transversality and semi-algebraicity. We refer to \cite{GM,GP,Hir,BR,BCR,Cos,KOS,Loj2,Loj3} for more details. Let $X,Y,Z,P$ be some $C^\infty$ manifolds. Let $S$ be a $C^\infty$ submanifold of $Y$. Let $f:X\rightarrow Y$ be a $C^\infty$  map. Denote $C^\infty(X,Y)$, the set of $C^\infty$ maps from $X$ to $Y$. 

\begin{defi} We say that $f$ is transverse to $S$ if $f(X)\cap S=\emptyset$ or for each $x\in f^{-1}(S)$, we have
$$d_xf(T_x X)+T_{f(x)}S=T_{f(x)}Y.$$
\end{defi}

Note that if $f$ is a submersion, then it is transverse to all submanifolds of $Y.$

\begin{prop}\label{codim} \cite{GP} If $f\pitchfork S$, then the inverse image $f^{-1}(S)$ of $S$ is a submanifold of $X$. Moreover, if $codim_Y S\leq dim X$ and $f(X)\cap S\not=\emptyset$, then
$$codim_X f^{-1}(S)=codim_Y S.$$
\end{prop}

If $codim_Y S>dim X$, then $f\pitchfork S$ if and only if $f(X)\cap S=\emptyset$; in this case, the transversality implies $codim_X f^{-1}(S)< codim_Y S$. 

\begin{prop}\label{classique} \cite{GP} Let $g: Y\rightarrow Z$ be a $C^\infty$ map, let $z\in Z$. Assume that $g$ is a submersion. Then the composed map $g\circ f$ is a submersion on $f^{-1}(g^{-1}(z))$ if and only if $f\pitchfork g^{-1}(z)$.
\end{prop}

\begin{defi} A semi-algebraic set of $\R^n$ is a subset of $\R^n$ defined by a finite sequence of polynomial equations and inequalities with real coefficients, or any finite union of such sets. In other words, the semi-algebraic subsets of $\R^n$ form the smallest class $SA_n$ of subsets of $\R^n$ such that:

1. If $P\in \R[x_1,\ldots, x_n]$, then $\{x\in \R^n : P(x) = 0\}\in SA_n$ and $\{x\in \R^n : P(x) > 0\}\in SA_n$.

2. If $A\in SA_n$ and $B\in SA_n$, then $A\cap B$, $A\cup B$ and $\R^n - A$ are in $SA_n$.
\end{defi}

\begin{theo}\label{DSA} \cite{BCR} Let $A\subset\R^n$ be a semi-algebraic set. If $A$ is dense in $\R^n$, then there exists a set $B$ which is semi-algebraic open dense in $\R^n$ such that $B\subset A$.
\end{theo}

Let us recall the following theorem on the semi-algebraic local triviality of semi-algebraic mappings.

\begin{theo}\label{Hardt}\cite{Har} Let $A\subset\R^n$ and $B\subset\R^m$ be some semi-algebraic sets. Let $f:A\to B$ be a semi-algebraic map. Then there is a $(\dim B-1)$-dimensional semi-algebraic subset $C$ of $B$ such that, for each connected
component $U$ of $B-C$, the restriction $f|_{f^{-1}(U)}\to U$ is a semi-algebraic fibration.
\end{theo}

We finish this part by recalling the following transversality theorem with parameters.

\begin{theo}\label{parametre}\cite{GM,GP} (Transversality theorem with parameters) Let $F:P\times X\rightarrow Y$ be a $C^\infty$ map. Set $f_p=F(p,.):X\rightarrow Y$. If $F\pitchfork S$, then the set
$$D=\{p\in P: f_p\pitchfork S\}$$
is open dense in $P$. Moreover, if $X,Y,S,P$ are semi-algebraic sets and if $F$ is a semi-algebraic map, then $D$ is also semi-algebraic.
\end{theo}
\begin{proof} The proof of openness and density of $D$ is done in \cite{GM,GP}. The method used also permits to prove that $D$ is semi-algebraic if $X,Y,S,P$ and $F$ are semi-algebraic.
\end{proof}

%------------------------------------------------------------------------------------------------------------------------------

\section{Generic horizontal critical set}

\hskip 0.5cm In this part, we study the generic dimension of $V_f$.

\begin{theo}\label{dimension}Generically, $V_f$ is smooth and of dimension $2$ or is empty, i.e., the set
\begin{equation}K_d=\{f\in\R_d[x]:V_f \text{ is smooth of dimension } 2 \text{ or is empty}\}\end{equation}
contains a semi-algebraic open dense set in $\R_d[x]$.
\end{theo}

\begin{proof} Let us write a polynomial $f\in\R_d[x]$ as follows
$$f=\displaystyle\alpha_0+\sum_{i=1}^4\alpha_ix_i+g$$
where $g$ does not have the linear part. Then

$X_1f=\alpha_1+X_1g,$

$X_2f=\alpha_2+\alpha_3x_1+\alpha_4x_3+X_2g.$

\noindent We identify $f\in\R_d[x]$ with its coefficient vector. Set

\begin{displaymath}\begin{array}{ccl}
$$L:\R_d[x]\times\R^4& \rightarrow & \R^2$$\\
$$(f,x) & \mapsto & ((X_1f)(x),(X_2f)(x)).$$
\end{array}\end{displaymath}
The Jacobian matrix of $L$ contains the following matrix

\begin{displaymath} (\displaystyle\frac{\partial L}{\partial\alpha_1},\frac{\partial L}{\partial\alpha_2})=
\left(\begin{array}{cccc}
$$1&0$$\\
$$0&1$$
\end{array}\right).\end{displaymath}
This implies that $L$ is a submersion. Consequently, $L$ is transverse to $\{0\}\subset\R^2.$ By Theorem \ref{parametre}, the set

\begin{equation}\label{Kd'}
K_d':=\{f\in\R_d[x]: \nabla^hf = L(f,.)\pitchfork \{0\}\}
\end{equation}
is semi-algebraic open dense in $\R_d[x]$. For each $f\in K_d'$, by Proposition \ref{codim} and by this transversality, $V_f$ is smooth and if $V_f\not=\emptyset$, then $codim V_f=codim_{\R^2}\{0\}=2,$ so $dim V_f=2$. Therefore $K_d'\subset K_d$. This ends the proof of the theorem.
\end{proof}

\begin{prop} \label{submersion} Generically, the map $\nabla^h f:\R^4\rightarrow\R^2$ is a submersion on $V_f$, i.e., the set
$$L_d:=\{f\in\R_d[x]:\nabla^hf \text{ is submersive on }V_f\}$$
contains a semi-algebraic open dense set in $\R_d[x]$.
\end{prop}

\begin{proof} For every $f\in K_d'$, the transversality condition $\nabla^hf\pitchfork \{0\}$ implies that $\nabla^hf$ is a submersion on $(\nabla^hf)^{-1}(\{0\})$, i.e., on $V_f$.
\end{proof}

If $f$ is an affine form $f=\displaystyle\sum_{i=1}^4 \alpha_ix_i+\alpha_0$, then $X_1f=\alpha_1$, $X_2f=\alpha_2+\alpha_3x_1+\alpha_4x_3.$ So if $\alpha_1\not=0$, then $V_f$ is empty. Let $K_1'':=\{f\in\R_1[x] : \alpha_1\not=0\}$. It is clear that $K_1''$ is semi-algebraic open dense in $\R_1[x]$ and $K_1''\subset K_1$. Then we have the following result for the affine case.

\begin{prop}\label{empty} Generically, $V_f$ is empty for $f\in\R_1[x]$, i.e., the set
\begin{equation}\label{K1}\underline K_1=\{f\in\R_1[x]: V_f \text{ is empty} \}\end{equation}
contains a semi-algebraic open dense set in $\R_1[x]$.
\end{prop}

%-------------------------------------------------------------------------------------------

\section{Horizontal curve in $V_f$}

\hskip 0.5cm From now on, we denote $X_{i_1i_2\ldots i_k}f:=X_{i_1}(X_{i_2}(\ldots(X_{i_k}f)\ldots))$ for $f\in\R_d[x]$. We still write a polynomial $f\in\R_d[x]$ under the following form
$$\displaystyle f=\sum_{i=1}^4\alpha_ix_i+\frac{1}{2}\sum_{i,j=1}^4\beta_{ij}x_ix_j+g, $$ 
where $g$ does not contain monomials of degree $1,2$. The goal of this section is to prove the following result.

\begin{theo}\label{noncontient} Generically, the set $V_f$ does not contain a horizontal curve which is contained in a fiber of $f$, i.e., the set
\begin{equation}\label{Md}M_d=\{f\in\R_d[x]: V_f\cap f^{-1}(t) \text{ does not contain a horizontal curve for any }t\}\end{equation}
contains a semi-algebraic open dense set in $\R_d[x]$.
\end{theo}

In the case $d=1$, the proof of Theorem \ref{noncontient} is trivial since $V_f=\emptyset$ for all $\dis f=\alpha_0+\sum_{i=1}^4\alpha_ix_i$ with $\alpha_1\not=0$. So assume that $d\geq 2$. First of all, we characterize the 
set of horizontal curves contained in the fibers of $f$.

Set \begin{equation}\label{gamma}\Gamma_f:=\{x\in V_f: T_xV_f\cap\Delta_x\not=\{0\}\}.\end{equation}

First of all, we show that $\Gamma_f$ is a smooth algebraic curve generically.

\begin{lem}\label{courbelisse} There exists a semi-algebraic open dense set $B_d\subset\R_d[x]$ such that for every $f\in B_d$, the set $\Gamma_f$ is a smooth algebraic curve or an empty set.
\end{lem}
\begin{proof} Consider the map
\begin{displaymath}\begin{array}{rcl}
$$\Theta:\R^4\times\R_d[x]&\rightarrow&\R^2\times M_2$$\\
$$(x,f)&\mapsto&\displaystyle(X_1f(x),X_2f(x),X_{11}f(x),X_{21}f(x),X_{22}f(x),X_{12}f(x)),$$
\end{array}\end{displaymath}
where $M_2\cong\R^4$ is the space of $2\times 2$ matrices. The Jacobian matrix of $\Theta$ contains the following matrix
\begin{displaymath} \displaystyle\frac{\partial\Theta}{\partial(\alpha_1,\alpha_2,\beta_{11},\beta_{12},\beta_{22},\alpha_3)}=
\left (\begin{array}{cccccc}
$$1 & 0 & x_1 & x_2 & 0 & 0$$ \\
$$0 & 1 & 0 & x_1 & x_2 & x_1$$ \\ 
$$0 & 0 & 1 & 0 & 0  & 0$$ \\
$$0 & 0 & 0 & 1 & 0 & 0$$ \\
$$0 & 0 & 0 & 0 & 1 & 0$$ \\
$$0 & 0 & 0 & 1 & 0 & 1$$ \\
\end{array}\right),
\end{displaymath}
whose rank is $6$. Therefore, $\Theta$ is a submersion and hence is transverse to $\{(0,0)\}\times\Sigma_2$ where $\Sigma_2$ is the algebraic subset of $M_2$, constituted of degenerate matrices. Denote by $0I_2$ the zero matrix in $M_2$. Note that $0I_2$ is the only singular point of $\Sigma_2$. By  Transversality theorem with parameters \ref{parametre}, the sets
$$B_d^1:=\{f\in\R_d[x]:\Theta(.,f)\pitchfork \{(0,0)\}\times (\Sigma_2-\{0I_2\})\},$$
$$B_d^2:=\{f\in\R_d[x]:\Theta(.,f)\pitchfork \{(0,0)\times 0I_2\}\}$$
are semi-algebraic open dense in $\R_d[x]$. For every $f\in B_d^2$, by Proposition \ref{codim}, we deduce that $(0,0)\times 0I_2\not\in Im\Theta(.,f)$. Let $f\in B_d:=B_d^1\cap B_d^2$. We have $\Theta(.,f)\pitchfork\{(0,0)\}\times (\Sigma_2-\{0I_2\})$ and $Im\Theta(.,f)$ does not contain the singular point $\{(0,0)\times 0I_2\}$. Finally, by Proposition \ref{codim}, the set $(\Theta(.,f))^{-1}(\{(0,0)\}\times \Sigma_2)$ is smooth and if it is not empty, we have
$$codim_{\R^4} (\Theta(.,f))^{-1}(\{(0,0)\}\times \Sigma_2)=codim_{\R^6} \{(0,0)\}\times \Sigma_2=3.$$
Hence $dim(\Theta(.,f))^{-1}(\{(0,0)\}\times \Sigma_2)=1.$ Note that $x\in(\Theta(.,f))^{-1}(\{(0,0)\}\times \Sigma_2)$ if and only if $d_x(\nabla^hf)|_{\Delta_x}$ is degenerate, hence if and only if $T_xV_f\cap\Delta_x\not=\{0\}$. This ends the proof of the lemma.

\end{proof}

\begin{re}\label{coro}{\rm Is is easy to check that the condition that $d_x(\nabla^hf)|_{\Delta_x}$ is degenerate gives
\begin{equation}\label{equagammaf}\Gamma_f=\{x\in\R^4:X_1f=X_2f=X_{11}fX_{22}f-X_{21}fX_{12}f=0\}.\end{equation}}
\end{re}

\begin{re}\label{MT}{\rm If $\deg f\le d$, then
\begin{enumerate}
\item[(i)] $\deg X_1f\leq d-1$, $\deg X_2f\leq d$,
\item[(ii)] $\deg X_{11}f\leq d-2$, $\deg X_{21}f\leq d-1$, $\deg X_{12}f\leq d-1$, $\deg X_{22}f\leq d$.
\end{enumerate}
Hence by \cite{Mil,Tho}, the number of connected components of $\Gamma_f$ is bounded by $c(d):=2(d-1)[4(d-1)-1]^3$. }
\end{re}

The following lemma permits to localize the horizontal curves contained in $V_f$.

\begin{lem}\label{conhori} Let $f\in B_d$. If $\gamma\subset V_f$ is a horizontal curve, then $\gamma\subset\Gamma_f$.
\end{lem}

\begin{proof} By the proof of Lemma \ref{courbelisse}, it follows that $B_d\subset K_d$, so $V_f$ is smooth of dimension $2$ or is empty. Assume that $V_f\not=\emptyset$. Let $\gamma(t)$ be a parametrization of $\gamma$. Since $\gamma(t)$ is almost everywhere differentiable and  $\gamma\subset V_f$, we have $\dot\gamma(t)\in T_{\gamma(t)}V_f$ for almost $t$. On the other hand, since $\gamma$ is horizontal, it follows that $\dot\gamma(t)\in\Delta_{\gamma(t)}$. Hence, for almost $t$, we have  $\dot\gamma(t)\in T_{\gamma(t)}V_f\cap\Delta_{\gamma(t)}$, so $T_{\gamma(t)}V_f\cap\Delta_{\gamma(t)}\not=\{0\}$. This implies that $\gamma(t)\in \Gamma_f$ for almost $t$. By the absolute continuousness of $\gamma(t)$ and the smoothness of $\Gamma_f$ from Lemma \ref{courbelisse}, the lemma follows.
\end{proof}

\begin{re}\label{Impot}{\rm Let $f\in B_d$, if $f$ is non constant on each connected component of $\Gamma_f$, then  $V_f\cap f^{-1}(t)$ does not contain a horizontal curve for every $t$. Moreover, $\Gamma_f$ does not contain a horizontal curve.}
\end{re}

By classical Morse theory, we have 

\begin{lem}\label{MorseC} There exists a semi-algebraic open dense set $J_d\subset\R_d[x]$ such that for every $f\in J_d$, the set of critical points of $f$
$$Cr(f):=\{ \nabla f=(X_1f,X_2f,X_3f,X_4f)=0\}$$
is finite.
\end{lem}

Now we have the following constrain on the intersection $V_f\cap f^{-1}(t)$.

\begin{lem}\label{Constrain} Let $f\in I_d=B_d\cap J_d$, then for every $t$, the intersection $V_f\cap f^{-1}(t)$ contains only curves (smooth or not) and points, i.e., $\dim(V_f\cap f^{-1}(t))\leq 1$.
\end{lem}

\begin{proof} Suppose that $\dim(V_f\cap f^{-1}(t))> 1$ for some $t$. Let $S\subset V_f\cap f^{-1}(t)$ be a smooth surface of dimension $2$. Since the number of critical points of $f$ is finite, we may assume that $S$ does not contain any critical point of $f$. By the proof of Lemma \ref{courbelisse}, it follows that $B_d\subset K_d$, so $\dim V_f=2$. For each $x\in S$, we have $T_x V_f=T_x S\subset T_x f^{-1}(f(x))$ and $\Delta_x\subset T_x f^{-1}(f(x))$. Note that $\dim\Delta_x=2$ and $\dim T_x f^{-1}(f(x))=3$, so $T_xV_f\cap\Delta_x\not=\{0\}$. Hence $x\in \Gamma_f$. It follows that $S\subset\Gamma_f$. This contradicts to the fact that $\Gamma_f$ is a smooth curve for every $f\in B_d$. The lemma is proved.
\end{proof}

If that $\Gamma_f$ has a connected component which is contained in a fiber of $f$, then we can calculate the tangent vector of this component as follows.

\begin{lem}\label{tangentvector} Let $f\in I_d=B_d\cap J_d$. Assume that $\gamma$ is a connected component of $\Gamma_f$ which is contained in a fiber of $f$. Then for $x\in\gamma$, 
\begin{equation}\label{xi}\xi_f(x)=-[X_{21}f(x)]X_1+[X_{11}f(x)]X_2\end{equation}
is a tangent vector of $\gamma$ at $x$.
\end{lem}

\begin{proof} Note that by Lemma \ref{Constrain} and by assumption, $\gamma$ is a component of $V_f\cap f^{-1}(t)$ for some $t$. Now $(d_x(\nabla^hf,f-t))(\xi_f(x))=$
\begin{displaymath} 
\left (\begin{array}{cccc}
$$X_{11}f & X_{21}f & X_{31}f & X_{41}f$$ \\
$$X_{12}f & X_{22}f & X_{32}f & X_{42}f$$ \\
$$0 & 0 & X_3f & X_4f$$ \\
\end{array}\right)_x.
\left (\begin{array}{c}
$$-X_{21}f$$ \\
$$X_{11}f$$ \\
$$0$$ \\
$$0$$ \\
\end{array}\right)_x=
\left (\begin{array}{c}
$$0$$ \\
$$X_{11}fX_{22}f-X_{12}fX_{21}f$$ \\
$$0$$ \\
\end{array}\right)_x.
\end{displaymath}
Then $(d_x(\nabla^hf,f-t))(\xi_f(x))=0$ by Remark \ref{coro}, for all $x\in \gamma$. This implies that $\xi_f(x)$ is a tangent vector of $V_f\cap f^{-1}(t)$ at $x$. Hence $\xi_f(x)$ is a tangent vector of $\gamma$ at $x$.
\end{proof}

%----------------------------------------------------------------------------------

The vector $\xi_f(x)$ does not give a tangent direction of a connected component of $\Gamma_f$ which is contained in a fiber of $f$ if $\xi_f(x)=0$. So we are going to show that for a generic polynomial, $\xi_f(x)\not=0$ on $\Gamma_f$ excepted at finitely many points.

Recall that $\Gamma_f=\{x\in\R^4:X_1f=X_2f=X_{11}fX_{22}f-X_{21}fX_{12}f\}.$ Let 
\begin{displaymath}\begin{array}{lll}
$$\Omega_f^1&:=&\{x\in \Gamma_f:X_{11}f=0\}$$\\
$$ &=&\{x\in\R^4: X_1f=X_2f=X_{11}fX_{22}f-X_{21}fX_{12}f=X_{11}f=0\}$$\\
$$ &=&\{x\in\R^4: X_1f=X_2f=X_{11}f=X_{21}fX_{12}f=0\}$$\\
$$ &=&\{x\in\R^4: X_1f=X_2f=X_{11}f=X_{21}f=0\}\cup$$\\
$$ &&\{x\in\R^4: X_1f=X_2f=X_{11}f=X_{12}f=0\}$$\\
$$ &=:&S_f^1\cup S_f^2,$$
\end{array}\end{displaymath}
\begin{displaymath}\begin{array}{lll}
$$\Omega_f^2&:=&\{x\in \Gamma_f:X_{21}f=0\}$$\\
$$ &=&\{x\in\R^4: X_1f=X_2f=X_{11}fX_{22}f-X_{21}fX_{12}f=X_{21}f=0\}$$\\
$$ &=&\{x\in\R^4: X_1f=X_2f=X_{21}f=X_{11}fX_{22}f=0\}$$\\
$$ &=&\{x\in\R^4: X_1f=X_2f=X_{21}f=X_{11}f=0\}\cup$$\\
$$ &&\{x\in\R^4: X_1f=X_2f=X_{21}f=X_{22}f=0\}$$\\
$$ &=:&S_f^1\cup S_f^3.$$
\end{array}\end{displaymath}
By using the same technique used in the proof of Lemma \ref{courbelisse}, we can prove that there exists a semi-algebraic open dense set $D_d\subset\R_d[x]$ such that for each $f\in D_d$, each of the sets $S_f^1,S_f^2,S_f^3$ are finite. Then, we have the following:

\begin{lem}\label{Omega} There exists a semi-algebraic open dense set $D_d\subset\R_d[x]$ such that for every $f\in D_d$, the set $\Omega_f=\Omega_f^1\cup\Omega_f^2$ is finite, so $\xi_f(x)\not=0$ for $x\not\in\Omega_f$. \end{lem}

Let 
\begin{equation}\label{Gf}Gf:=X_{11}fX_{22}f-X_{21}fX_{12}f.\end{equation}
Then $\Gamma_f=\{X_1f=X_2f=Gf=0\}$. 

Now we are able to describe numerically the condition of tangency of $\Gamma_f$ to a fiber of $f$.

\begin{lem}\label{tangency} Let $f\in I_d\cap D_d$. Let $x\in\Gamma_f-(\Omega_f\cup Cr(f))$ where $Cr(f)$ is still the set of critical points of $f$, which is finite. Then $[d_x(Gf)](\xi_f(x))=0$ if and only if $\Gamma_f$ is tangent to the fiber $f^{-1}(f(x))$ at $x$.
\end{lem}
\begin{proof} Note that $\xi_f(x)\not=0$ and $\xi_f(x)$ is a tangent vector of $V_f\cap f^{-1}(f(x))$ at $x$. 

Suppose that $[d_x(Gf)](\xi_f(x))=0$. Let us prove that $\xi_f(x)$ is a tangent vector of $\Gamma_f$ at $x$. We have
$[d_x(X_1f,X_2f,Gf)](\xi_f(x))=$
\begin{displaymath}
=\left(
\begin{array}{cccc}
$$X_{11}f& X_{21}f& X_{31}f& X_{41}f$$\\
$$X_{12}f& X_{22}f& X_{32}f& X_{42}f$$\\
$$X_1G& X_2G& X_3G& X_4G$$\\
\end{array}\right)_x\left(
\begin{array}{c}
$$-X_{21}f$$\\
$$X_{11}f$$\\
$$0$$\\
$$0$$\\
\end{array}\right)_x=\left(
\begin{array}{c}
$$0$$\\
$$Gf(x)$$\\
$$[d_x(Gf)](\xi_f(x))$$\\
\end{array}\right).
\end{displaymath}
Hence $[d_x(X_1f,X_2f,Gf)](\xi_f(x))=0$, which means that $\xi_f(x)$ is a tangent vector of $\Gamma_f$ at $x$. Therefore, $\Gamma_f$  is tangent to the fiber $f^{-1}(f(x))$ at $x$. 

Now suppose that $\Gamma_f$ is tangent to the fiber $f^{-1}(f(x))$ at $x$. Let $\eta$ be a non-zero tangent vector of $\Gamma_f$ at $x$, then $[d_x(Gf)](\eta)=0$. Since $\Gamma_f\subset V_f$, then $\eta\in T_xV_f$. On the other hand, by the assumption, it follows that $\eta\in T_xf^{-1}(f(x)).$ Consequently, $\eta\in T_x(V_f\cap f^{-1}(f(x)))$. Moreover, the two non-zero vectors $\eta$ and $\xi_f(x)$ have to be linearly dependent. This implies that $[d_x(X_1f,X_2f,Gf)](\xi_f(x))=0$.
\end{proof}

%------------------------------------------------------------------------------------------------------------------------

For $f\in\R_d[x]$, we define the norm 
$$\displaystyle ||f||:=(\sum_i c_i^2)^{1/2}$$
where $c_i$ is a coefficient of $f$ and the sum is taken over the set of coefficients of $f$. Then the distance between $f,h\in\R_d[x]$ is defined by $||f-h||$. 

Let $\pi:\R^4\times\R_d[x]\rightarrow\R_d[x]$ be the projection map. Let $S:=\Theta^{-1}(\{(0,0)\}\times\Sigma_2)$ where $\Theta$ is defined in Lemma \ref{courbelisse} and $\Sigma_2$ is still the algebraic subset of $M_2$, constituted of degenerate matrices. Note that if $(x,f)\in S$, then $x\in\Gamma_f$. Let $\pi_S$ be the restriction of $\pi$ on $S$, then a fiber $(\pi_S)^{-1}(f)$ of $\pi_S$ is just $\Gamma_f\times\{f\}$. By Theorem \ref{Hardt}, there exists a semi-algebraic subset $Z$ of codimension $1$ of $\R_d[x]$ such that $f$ is a semi-algebraic locally trivial fibration over each connected component of $T_d:=\R_d[x]-Z$. It is clear that $T_d$ contains a semi-algebraic open dense subset of $\R_d[x]$. Let $A_d\subset I_d\cap T_d$ such that $A_d$ is semi-algebraic open dense.  Now fix a polynomial $f\in A_d$. The set $V_f$ is algebraic, smooth, $dim V_f=2$ or $V_f$ is empty, moreover $\Gamma_f$ is algebraic, smooth, $dim\Gamma_f=1$ or $\Gamma_f$ is empty. Suppose that $\Gamma_f\not=\emptyset$.

For each $f\in A_d\cap D_d$, since $\Gamma_f$ is an algebraic set, it has only a finite number of connected components $\Gamma^1_f,\ldots,\Gamma^s_f$, this number is bounded uniformly by $c(d)$ where $c(d)$ is the constant defined in Remark \ref{MT}. Let $\kappa(f)$ be the number of connected components of $\Gamma_f$ such that each of them is contained in a fiber of $f$ (not necessarily the same) and let $\lambda(f)$ be the number of connected components of $\Gamma_f$ which are not contained in any fiber of $f$. So $\lambda(f)+\kappa(f)\leq c(d)$. Now let us fix a polynomial $f\in A_d\cap D_d$. Suppose that $\kappa(f)>0$. Denote $\epsilon:=(\alpha,\beta,\gamma)$ and set
\begin{equation}\label{fep}\displaystyle f_{\epsilon}:=f+\alpha x_2+\frac{\beta}{2}x_2^2+\gamma x_4,\end{equation}
a perturbation of $f$ by a form of degree $2$. Let $|\epsilon|:=\max\{|\alpha|,|\beta|,|\gamma|\}$ be the "size" of the perturbation. If $\epsilon$ is small enough, by the triviality given by Theorem \ref{Hardt}, for each connected component $\theta$ of $\Gamma_f$, there exists a connected component $\theta_\epsilon$ of $\Gamma_{f_\epsilon}$ which is close to $\theta$. We call $\theta_\epsilon$ the connected component \textbf{corresponding} to $\theta$.

The following lemma is the key to prove the density in the theorem \ref{noncontient}.

\begin{lem}\label{redu} There exists $\epsilon=(\alpha,\beta,\gamma)$ such that $\lambda(f_\epsilon)>\lambda(f)$. Moreover $$\lambda(f_{t\epsilon})=\lambda(f_{(t\alpha,t\beta,t\gamma)})>\lambda(f)$$
for every $t\in(0,1]$.
\end{lem}

\begin{proof} We write $\displaystyle\Gamma_f=\bigcup_{i=1}^{\lambda(f)}\Gamma_i\cup\bigcup_{i=\lambda(f)+1}^{\lambda(f)+\kappa(f)}\Gamma_i,$ where

\begin{enumerate}
\item[(i)] $\Gamma_f^1,\ldots,\Gamma_f^{\lambda(f)}$ are the connected components of $\Gamma_f$ which are not contained in any fiber of $f$,
\item[(ii)] $\Gamma_f^{\lambda(f)+1},\ldots,\Gamma_f^{\lambda(f)+\kappa(f)}$ are the connected components of $\Gamma_f$ which are contained in a fiber of $f$ (not necessarily the same).
\end{enumerate}

By Lemma \ref{tangency}, it follows that $\xi_f(x)\not=0$ and $[d_x(Gf)](\xi_f(x))=0$ for every $x\in\dis\bigcup_{i=\lambda(f)+1}^{\lambda(f)+\kappa(f)}\Gamma_i-(\Omega_f\cup Cr(f)).$

Let us fix $\theta=\Gamma_f^{\lambda(f)+1}$, a connected component of $\Gamma_f$ which is contained in a fiber of $f$. We look for a perturbation $f_\epsilon$ satisfying the following conditions:

\begin{enumerate}
\item[(c1)] The connected component $\theta_\epsilon$ of $\Gamma_{f_\epsilon}$ corresponding to $\theta$ is not contained in any fiber of $f_\epsilon$.
\item[(c2)] The connected components of $\Gamma_{f_\epsilon}$ corresponding to $\Gamma_f^1,\ldots,\Gamma_f^{\lambda(f)}$ are not contained in any fiber of $f_\epsilon$ neither.
\end{enumerate}

Firstly we establish c1). By Lemma \ref{Omega}, the set $\Omega_f$ is finite. Since $f\in J_d$, by Lemma \ref{MorseC}, $f$ is a Morse function, then $Cr(f)$ is finite. Consequently, $\theta-(\Omega_f\cup Cr(f))\not=\emptyset$. Let us take a point $b=(b_1,b_2,b_3,b_4)\in\theta-(\Omega_f\cup Cr(f))$. Let $\epsilon=(\alpha,\beta,\gamma)$ be a solution of the following system of equations:
\begin{equation}\label{direction}\left\{
\begin{array}{r}
\displaystyle\alpha+b_2\beta+b_3\gamma=0\\
\beta+b_1\gamma=0\\
\end{array}\right.\Leftrightarrow
\left\{
\begin{array}{l}
\displaystyle\alpha=-b_2\beta-b_3\gamma=(b_1b_2-b_3)\gamma\\
\beta=-b_1\gamma\\
\end{array}\right.
\end{equation}
such that $\gamma\not=0.$ Note that by choosing $\gamma$ small, we can make the size of $\epsilon$ arbitrary small. Recall that $\displaystyle X_3=\frac{\partial}{\partial x_3}, X_4=\frac{\partial}{\partial x_4}$. Now we have
\begin{equation}
\begin{array}{ccll}
$$X_if_\epsilon&=&X_if \hskip 1.1cm \text{ for } i\in\{1,3,4\},$$\\
$$X_2f_\epsilon&=&X_2f+\alpha+\beta x_2+\gamma x_3,$$\\
$$X_{ij}f_\epsilon&=&X_{ij}f \hskip 1cm \text{ for }(i,j)\in\{1,4\}\times\{1,4\},(i,j)\not=(2,2),(3,2),$$\\
$$X_{22}f_\epsilon&=&X_{22}f+\beta+\gamma x_1.$$
\end{array}
\end{equation}
We remark the following properties.

\begin{cla}\label{cl1}The perturbation $f_\epsilon$ fixes the point $b$ in $\Gamma_f$, i.e., $b\in\Gamma_{f_\epsilon}$.
\end{cla}

\begin{proof} We have 
$$Gf_\epsilon=X_{11}f_\epsilon X_{22}f_\epsilon-X_{21}f_\epsilon X_{12}f_\epsilon=X_{11}f(X_{22}f+\beta+\gamma x_1)-X_{21}fX_{12}f=Gf+(\beta+\gamma x_1)X_{11}f.$$
By our choice of $\epsilon$, it is easy to see that $X_1f_\epsilon(b)=X_2f_\epsilon(b)=Gf_\epsilon(b)=0$. This proves the claim.
\end{proof}

\begin{cla}\label{cl2} The connected component $\theta_\epsilon$ of $\Gamma_{f_\epsilon}$ which contains $b$ is not  contained in any fiber of $f_\epsilon$. In fact $\theta_\epsilon\pitchfork_b f_\epsilon^{-1}(f_\epsilon(b))$.
\end{cla}

\begin{proof} First of all, we prove that $b\not\in \Omega_{f_\epsilon}\cup Cr(f_\epsilon)$. Since $X_if_\epsilon(b)=X_if(b)$ for $i=1,2,3,4$, we deduce that $\nabla f_\epsilon(b)=\nabla f(b)$ where $\nabla f_\epsilon$ and $\nabla f$ denote respectively the (Riemannian) gradient of $f_\epsilon$ and $f$. Since $b\not\in Cr(f)$, $b$ is not a (Riemannian) critical point of $f$ and hence $b$ is not a (Riemannian) critical point of $f_\epsilon$. This implies that $f_\epsilon^{-1}f_\epsilon((b))$ is smooth at $b$. Moreover $T_bf_\epsilon^{-1}f_\epsilon((b))=T_bf^{-1}f((b))$. On the other hand, since $X_{ij}f_\epsilon(b)=X_{ij}f(b)$ $(i=1,2,3,4,j=1,2)$ by our choice of $\epsilon$, it follows that $d_b(\nabla^hf)=$
\begin{displaymath}
\begin{pmatrix}
$$X_{11}f&X_{21}f&X_{31}f&X_{41}f$$\\
$$X_{12}f&X_{22}f&X_{32}f&X_{42}f$$\\
\end{pmatrix}_b=
\begin{pmatrix}
$$X_{11}f_\epsilon&X_{21}f_\epsilon&X_{31}f_\epsilon&X_{41}f_\epsilon$$\\
$$X_{12}f_\epsilon&X_{22}f_\epsilon&X_{32}f_\epsilon&X_{42}f_\epsilon$$\\
\end{pmatrix}_b=d_b(\nabla^hf_\epsilon).
\end{displaymath}

Hence $T_bV_f=Kerd_b(\nabla^hf)=Kerd_b(\nabla^hf_\epsilon)=T_bV_{f_\epsilon}.$ Note that $X_{11}f_\epsilon(b)=X_{11}f(b)\not=0$ and $X_{21}f_\epsilon(b)=X_{21}f(b)\not=0$, consequently, $b\not\in\Omega_{f_\epsilon}.$ So $b\not\in \Omega_{f_\epsilon}\cup Cr(f_\epsilon)$.

Now by Lemma \ref{tangency}, it is sufficient to show that $[d_b(Gf_\epsilon)](\xi_{f_\epsilon}(b))\not=0$. Note that
$$Gf_\epsilon=Gf+(\beta+\gamma x_1)X_{11}f,$$
$$\xi_{f_\epsilon}=-(X_{21}f_\epsilon)X_1+(X_{11}f_\epsilon)X_2=-(X_{21}f)X_1+(X_{11}f)X_2=\xi_f$$
Hence
\begin{equation}\label{calculdif}\begin{array}{lll}
$$[d_b(Gf_\epsilon)](\xi_{f_\epsilon}(b))&=&[d_b(Gf+(\beta+\gamma x_1)X_{11}f)](\xi_f(b))$$\\
$$&=&[d_b(Gf)](\xi_f(b))+[d_b\{(\beta+\gamma x_1)X_{11}f\}](\xi_f(b))$$\\
$$&=&[d_b(Gf)](\xi_f(b))+[(\beta+\gamma x_1)d(X_{11}f)$$\\
$$&&+(X_{11}f)d\{(\beta+\gamma x_1)\}]_b(\xi_f(b))$$\\
$$&=&[d_b(Gf)](\xi_f(b))+(\beta+\gamma b_1)[d_b(X_{11}f)](\xi_f(b))$$\\
$$&&+(X_{11}f(b))[d_b(\beta+\gamma x_1)](\xi_f(b)).$$\\
$$&=&[d_b(Gf)](\xi_f(b))+(\beta+\gamma b_1)[d_b(X_{11}f)](\xi_f(b))$$\\
$$&&-\gamma[X_{11}f(b)].[X_{21}f(b)].$$\\
\end{array}\end{equation}
Since $[d_b(Gf)](\xi_f(b))=0$ by our initial setting and $(\beta+\gamma b_1)=0$ by our choice of $\epsilon$, it follows from Lemma \ref{Omega} that $[d_b(Gf_\epsilon)](\xi_{f_\epsilon}(b))=-\gamma[X_{11}f(b)].[X_{21}f(b)]\not=0$ for $\epsilon\not=0$. This ends the proof of the claim.
\end{proof}

So with Claim \ref{cl2}, we have established c1). Now we try to establish c2). To do so, the following claim is sufficient.

\begin{cla}\label{cl3} For each $i=1,\ldots,\lambda(f)$, let $\Gamma_{f_\epsilon}^i$ be the connected component of $\Gamma_{f_\epsilon}$ corresponding to $\Gamma_f^i$. There exists a constant $0<\delta$ such that if $|\epsilon|\leq\delta$, then $\Gamma_{f_\epsilon}^i$ is not contained in any fiber of $f_\epsilon$.
\end{cla}

\begin{proof} Since $\Gamma_f^i$ is not contained in any fiber of $f$, it follows that $\Gamma_f^i$ is almost everywhere transverse to the fibers of $f$. By Lemma \ref{Omega}, the set $\Omega_f\cup Cr(f)$ is finite. Hence, we can choose a point $a_i\in\Gamma_f^i-(\Omega_f\cup Cr(f))$ such that $\Gamma_f^i\pitchfork_{a_i} f^{-1}(f(a_i))$. By Lemma \ref{tangency}, $[d_{a_i}(Gf)](\xi_f(a_i))\not=0$. Let 
$$\displaystyle m_1:=\min_{i=1,\ldots,s_1} |[d_{a_i}(Gf)](\xi_f(a_i))|>0.$$
Let $\zeta>0$ such that $B(a_i,\zeta)\cap\{|[d(Gf)](\xi_f)|\leq\displaystyle\frac{3m_1}{4}\}=\emptyset$ where $B(a_i,\zeta):=\{x\in \R^4: ||x-a_i||\leq\zeta\}.$ Set 
$$\displaystyle m_2:=\max_{i=1,\ldots,s_1;x\in B(a_i,\zeta)} |[d_x(X_{11}f)](\xi_f(x))|<+\infty,$$
$$\displaystyle m_3:=\max_{i=1,\ldots,s_1;x\in B(a_i,\zeta)} |X_{11}f(b)|.|X_{21}f(b)|<+\infty,$$
$$\displaystyle m_4:=\max_{i=1,\ldots,s_1;x\in B(a_i,\zeta)} ||x-b||<+\infty.$$
%where $\B$ is a ball center at the origin which contains $b$ and $B(a_i,\zeta)$. 
Note that $m_3,m_4>0$. Let $r>0$ such that 
\begin{displaymath}r\le\left\{\begin{array}{ll}
$$\displaystyle\min\{\frac{m_1}{4m_2m_4},\frac{m_1}{4m_3}\} &\text{ if } m_2\not=0$$\\
$$\displaystyle\frac{m_1}{4m_3} &\text{ if } m_2=0.$$\\
\end{array}\right.\end{displaymath}
and such that for $|\epsilon|<r$, $f_\epsilon\in A_d\cap D_d$ and the intersection $B(a_i,\zeta)\cap\Gamma_{f_\epsilon}^i$ is a curve.
Since the set $\Omega_{f_\epsilon}\cup Cr(f_\epsilon)$ is finite, we can choose a point $\sigma_i=(\sigma_{i1},\sigma_{i2},\sigma_{i3},\sigma_{i4})\in B(a_i,\zeta)\cap\Gamma_{f_\epsilon}^i$ such that $\sigma_i\not\in\Omega_{f_\epsilon}\cup Cr(f_\epsilon)$. Let us prove that $[d_{\sigma_i}(Gf_\epsilon)](\xi_{f_\epsilon}(\sigma_i))\not=0$. In fact, by the same computation as (\ref{calculdif}), we have
\begin{displaymath}\begin{array}{lll}
$$[d_{\sigma_i}(Gf_\epsilon)](\xi_{f_\epsilon}(\sigma_i))&=&[d_{\sigma_i}(Gf)](\xi_f(\sigma_i))+(\beta+\gamma \sigma_{i1})[d_{\sigma_i}(X_{11}f)](\xi_f(\sigma_i))$$\\
$$&&-\gamma[X_{11}f(\sigma_i)].[X_{21}f(\sigma_i)].$$\\
\end{array}\end{displaymath}
So 
\begin{displaymath}\begin{array}{lll}
$$|[d_{\sigma_i}(Gf_\epsilon)](\xi_{f_\epsilon}(\sigma_i))|&\geq&|[d_{\sigma_i}(Gf)](\xi_f(\sigma_i))|-|(\beta+\gamma \sigma_{i1})|.|[d_{\sigma_i}(X_{11}f)](\xi_f(\sigma_i))|$$\\
$$&&-|\gamma|.|X_{11}f(\sigma_i)|.|X_{21}f(\sigma_i)|.$$
\end{array}\end{displaymath}
We note the following facts:

i) Since $\sigma_i\in B(a_i,\zeta)$ and $B(a_i,\zeta)\cap\{|[d{\sigma_i}(Gf)](\xi_f)|\leq\displaystyle\frac{3m_1}{4}\}=\emptyset$, it follows that 
$$|[d_{\sigma_i}(Gf)](\xi_f(\sigma_i))|>\displaystyle\frac{3m_1}{4}.$$

ii) $|(\beta+\gamma \sigma_{i1})|.|[d_{\sigma_i}(X_{11}f)](\xi_f(\sigma_i))|=|(-b_1+\sigma_{i2})\gamma|.|[d_{\sigma_i}(X_{11}f)](\xi_f(\sigma_i))|\leq m_4 r m_2\leq\displaystyle\frac{m_1}{4}.$

iii) $|\gamma|.|X_{11}f(\sigma_i)|.|X_{21}f(\sigma_i)|\leq r m_3\leq\displaystyle\frac{m_1}{4}.$

\noindent These imply that 
$$|[d_{\sigma_i}(Gf_\epsilon)](\xi_{f_\epsilon}(\sigma_i))|\geq\displaystyle\frac{3m_1}{4}-\frac{m_1}{4}-\frac{m_1}{4}=\frac{m_1}{4}>0.$$
Then by Lemma \ref{tangency}, we deduce that $\Gamma_{f_\epsilon}^i\pitchfork_{\sigma_i} f^{-1}(f_\epsilon(\sigma_i)).$ This ends the proof of the claim.

\end{proof}

By contruction, $\lambda(f_\epsilon)\geq\lambda(f)+1$. Moreover, $\lambda(f_{t\epsilon})\geq\lambda(f)+1$ for every $t\in(0,1]$. Hence, the lemma is proved.
\end{proof}

Lemma \ref{redu} shows that, after a small appropriate perturbation $f_\epsilon$, the number of connected components of $\Gamma_{f_\epsilon}$, which are not contained in any fiber of $f_\epsilon$, increases. So we obtain an important corollary as follows.

\begin{coro}\label{density} The intersection $M_d\cap A_d\cap D_d$, where $M_d$ is defined by (\ref{Md}), is dense in $\R_d[x]$.
\end{coro}
\begin{proof} Since the number $\lambda(f)$ is bounded from above by $c(d)$, repeated applications of Lemma \ref{redu} give a polynomial $h\in M_d$ which can be chosen arbitrarily close to $f$. The corollary follows.
\end{proof}

The following lemma proves the semi-algebraicity of the set $M_d\cap A_d\cap D_d$.

\begin{lem}\label{semia} The set $M_d\cap A_d\cap D_d$, where $M_d$ is defined by (\ref{Md}), is a semi-algebraic set.
\end{lem}
\begin{proof} For $f\in A_d\cap D_d$, denote $\Gamma_f^1,\ldots,\Gamma_f^{k_f}$ be the connected components of $\Gamma_f$ which are semi-algebraic sets. Then
$$\begin{array}{lll}
$$M_d\cap A_d\cap D_d=\{f\in A_d\cap D_d:&\exists \sigma_i\in\Gamma_f^i-(\Omega_f\cup Cr(f)):$$\\
$$&[d_{\sigma_i}(Gf)](\xi_f(\sigma_i))\not=0,i=1,\ldots,k_f\}.$$
\end{array}$$
This set is clearly a semi-algebraic set by Tarski-Seidenberg principle \cite{BR,BCR,Cos} and Theorem \ref{Hardt}.
\end{proof}

Now we have all necessary ingredients to prove the theorem \ref{noncontient}. By Corollary \ref{density} and lemma \ref{semia}, it follows that $M_d\cap A_d\cap D_d$ is semi-algebraic and dense in $\R_d[x]$. Then by Theorem \ref{DSA}, $M_d\cap A_d\cap D_d$ contains a semi-algebraic open dense set in $\R_d[x]$. Consequently, $M_d$ contains a semi-algebraic open dense set in $\R_d[x]$. This completes the proof of the theorem \ref{noncontient}.

%--------------------------------------------------------------------------------------------------------------

\section{Trajectories of horizontal gradient}
In this section, we prove that for a generic polynomial $f$, every trajectory of $\nabla^hf$ in a given compact set $B$ always has a limit. Precisely, we will prove the following.

\begin{theo}\label{havelimit} Generically, the trajectories of $\nabla^hf$ in $B$ have a limit, i.e., the set
\begin{equation}\label{Nd}N_d=\{f\in\R_d[x]: \text{ the trajectories of }\nabla^hf \text{ have a limit on } \partial B\cup V_f\}\end{equation}
contains a semi-algebraic open dense set in $\R_d[x]$.
\end{theo}

Let $d_{V_f}(.)$ be the distance function to $V_f$ with respect to the usual Euclidean metric on $\R^4$. By \L ojasiewicz inequality (\cite{Loj2}), there exist some constants $C_1,C_2>0,\gamma\geq1\geq\beta>0$ such that for all $x\in B$, we have
$$C_1d_{V_f}^\gamma(x)\leq||\nabla^h f(x)||\leq C_2d_{V_f}^\beta(x).$$

We have the following lemma.
\begin{lem}\label{expo} For $f\in L_d$, where $L_d$ is defined in Proposition \ref{submersion}, we can choose the exponents $\gamma=\beta=1$ in the \L ojasiewicz inequality (\ref{Lojas}), i.e., there exist some constants $0<C_1\leq C_2$ such that for all $x\in B$, we have
\begin{equation}\label{Lojas}
C_1d_{V_f}(x)\leq||\nabla^h f(x)||\leq C_2d_{V_f}(x).
\end{equation}
\end{lem}

\begin{proof} The idea of proof is similar to Lemma 5.15 in \cite{DKO}. By Proposition \ref{submersion}, for $f\in L_d$, $\nabla^h f$ is a submersion on $V_f$. This implies that for each $x\in V_f$, the Jacobian matrix of $\nabla^hf:\R^4\rightarrow \R^2$ is of constant rank $2$. By constant rank theorem, for all $x\in V_f,$ there exist a diffeomorphism $u:U_1^x\rightarrow U^x$ from a neighborhood $U_1^x$ of $0\in\R^4$ on a neighborhood $U^x$ of $x$, $u(0)=x$, and a diffeomorphism $w:W\rightarrow W_1$ from a neighborhood $W$ of $0\in\R^2$ on a neighborhood $W_1$ of $0\in\R^2$, $w(0)=0$ such that
$$A(y_1,y_2,y_3,y_4)=(y_1,y_2)$$
where $A=w\circ\nabla^hf\circ u$. It is clear that $u^{-1}(V_f\cap U^x)=\{y_1=y_2=0\}\cap U_1^x$ and that
\begin{equation}\label{2*}||A(y)||=||(y_1,y_2)||=dist(y,V_1f)\end{equation}
where $dist(.,V_1f)$ is the distance to $V_1f$ with respect to the standard metric on $\R^4$. Since $u$ and $w$ are diffeomorphism (defined on the convex sets), they are bi-Lipschitz, i.e., $u,v$ and $u^{-1},v^{-1}$ are Lipschitz. Hence, by (\ref{2*}), there exist some constants $0<C_1^x\leq C_2^x<+\infty$ such that for all $z\in U^x$, we have
\begin{equation}\label{3*}C_1^xd_{V_f}(z)\leq||\nabla^hf(z)||\leq C_2^xd_{V_f}(z).\end{equation}
By compactness of $B\cap V_f$, we can find a finite cover of $B\cap V_f$ by some open sets $U^{x_i}$, $i=1,\ldots,k$, such that (\ref{3*}) holds in each $U^{x_i}$. Note that there exists $r>0$ such that for all $z\in\tilde{B}=B-\displaystyle\bigcup_i U^{x_i}$, we have $d_{V_f}(z)\geq r,$ so let
$$\displaystyle\tilde{C}_1=\min_{\tilde{B}}\frac{||\nabla^hf||}{d_{V_f}},\tilde{C}_2=\max_{\tilde{B}}\frac{||\nabla^hf||}{d_{V_f}}.$$ 
Then
$$\tilde{C}_1d_{V_f}(z)\leq||\nabla^hf(z)||\leq\tilde{C}_2d_{V_f}(z)$$
for all $z\in\tilde{B}$. Finally, let
$$C_1=\min\{C_1^{x_1},\ldots,C_1^{x_k},\tilde{C}_1\},C_2=\max\{C_2^{x_1},\ldots,C_2^{x_k},\tilde{C}_2\}.$$
Then for all $z\in B$, we have
$$C_1d_{V_f}(z)\leq||\nabla^h f(z)||\leq C_2d_{V_f}(z).$$
The lemma is proved.
\end{proof}

Recall that $g_{SR}=dx_1^2+dx_2^2$ is a sub-Riemannian metric on $\Delta$. Let $\delta_{SR}=d_{V_f}^2g_{SR}$. Note that the metrics $\delta_{SR}$ depends on the polynomial $f$. Moreover, it is degenerate on $V_f$.

Let $<<.,.>>$ and $||.||$ denote respectively the scalar product and the norm for the metric $g_{SR}$, let $<.,.>$ and $|.|$ denote respectively the scalar product and the norm for the metric $\delta_{SR}$. For a horizontal curve $\alpha:(t_1,t_2)\to\R^4$, its length with respect to $g_{SR}$ and $\delta_{SR}$, denoted respectively by $l_g(\alpha(t))$ and $l_\delta(\alpha(t))$ are given by the following formulas.

%The sub-Riemannian distances between two points $x,y\in \R^4$ for the metric $g_{SR}$ (denoted by $\dist(x,y)$) and for the metric $\delta_{SR}$ (denoted by $\dist_1(x,y)$) are given respectively by the following formulas
\begin{equation}\label{lg}
\displaystyle l_g(\alpha(t))=\int_0^T||\dot\alpha(t)||_{\alpha(t)}dt,\end{equation}
\begin{equation}\label{ldelta}\begin{array}{rcl}
$$l_\delta(\alpha(t))&=&\dis\int_0^T|\dot\alpha(t)|_{\alpha(t)}dt$$\\
$$&=&\dis\int_0^T d_{V_f}(\alpha(t))||\dot\alpha(t)||_{\alpha(t)}dt.$$\end{array}\end{equation}

%It is clear that $\dist$ is a distance on $\R^4$ while $\dist_1$ it not necessarily a distance on $\R^4$ since it may be degenerate on $V_f$. However the function $\dist_1(x,y)$ is well defined for $x,y\in V_f$.

The horizontal gradient for the metric $\delta_{SR}$ can be computed from the one for the metric $g_{SR}$, this is given by the following lemma.

\begin{lem}\label{metriccompare} Let $^\delta\nabla f=(u_1,u_2,u_3,u_4)$ and $^\delta\nabla^h f=(u_1,u_2)$, be respectively the gradient and the  horizontal gradient of $f$ with respect to the metric $\delta:=\delta_{SR}+\omega_3^2+\omega_4^2$ where $\omega_3=dx_3-x_1dx_2,$ $\omega_4=dx_4-x_3dx_2$. Then, in the coordinate system $\{X_1,X_2,X_3,X_4\}$, we have

\begin{equation}
\left \{
\begin{array}{ll}
$$\displaystyle u_i=\frac{X_i f}{d_{V_f}^2}, (i=1,2)$$ \\
$$u_i=X_i f, (i=3,4)$$
\end{array}\right .
\end{equation}
Which mean
$\displaystyle ^\delta\nabla^h f=\frac{\nabla^h f}{d_{V_f}^2}$ and $\displaystyle |^\delta \nabla^h f|=\frac{||\nabla^h f||}{d_{V_f}}.$
\end{lem}
\begin{proof} The proof is similar to the proof of Lemma 5.8 in \cite{DKO}.
\end{proof}

The following lemma is similar to Theorem 5.10 in \cite{DKO}.
\begin{lem}\label{finitelength} Let $\alpha:(t_1,t_2)\to B$ is a trajectory of $\nabla^hf$. Then for the metric $\delta$, the length of $\alpha$ is bounded by $\displaystyle \frac{|f(\alpha(t_2))-f(\alpha(t_1))|}{C_1}$ where $C_1$ is the constant (depended on $B$) in the \L ojasiewicz inequality (\ref{Lojas}).
\end{lem}
\begin{proof} The proof is obtained by the same way as the proof of Theorem 5.10 in \cite{DKO}.
\end{proof}

The following lemma is similar to Proposition 5.1 in \cite{DKO}.
\begin{lem}\label{growth} If $x(t)$ is not a constant trajectory of $\nabla^hf$, then $\dis\frac{d}{dt}(f(x(t)))>0$ which means that $f$ is non constant on $x(t)$. So the there are no closed trajectories of $\nabla^hf$.
\end{lem}
\begin{proof} The proof is analogue to the proof of Proposition 5.1 in \cite{DKO}.
\end{proof}

Now for $a\in V_f-\Gamma_f$, let $U$ be an open neighborhood of $a$ such that $U\cap \Gamma_f=\emptyset$, let $Y,Z$ be some open neighborhoods of $0\in\R^2$, let 
\begin{equation}\begin{array}{rcll}
\varphi:&Y&\to& V_f\cap U\\
&(y_1,y_2)&\mapsto& \varphi(y_1,y_2)
\end{array}\end{equation}
be a smooth local parameterization of $V_f$ at $a$ such that $\varphi(0)=a$. Set 
\begin{equation}\begin{array}{rcll}
\Phi:&Y\times \R^2&\to& U\\
&y=(y_1,y_2,y_3,y_4)&\mapsto& \varphi(y_1,y_2)+y_3X_1(\varphi(y_1,y_2))+y_4X_2(\varphi(y_1,y_2))
\end{array}\end{equation}

We have the following lemma.

\begin{lem}\label{diffeolocal} The map $\Phi$ is a local diffeomorphism at $0\in\R^4$.
\end{lem}
\begin{proof} Note that $(d\Phi)_0(\R^2\times\{0\})=\{(d\varphi)_0(\R^2)\}\times\{0\}=T_aV_f\times \{0\}$ and $(d\Phi)_0(\{0\}\times\R^2)=\{0\}\times\Delta_a$. On the other side, $T_aV_f\pitchfork\Delta_a$ since $a\not\in\Gamma_f$. Hence $(d\Phi)_0$ is surjective which implies that it is also bijective. The lemma follows.
\end{proof}

\begin{re}\label{equiv}{\rm Let $dist$ and $dist_E$ be respectively the Riemannian distance with respect to the metric $g=dx_1^2+dx_2^2+\omega_3^2+\omega_4^2$ and the usual Euclidean metric $g_E$ on $\R^4$. From the property that all scalar products on a finite  dimensional vector space are equivalent, it follows that $dist$ and $dist_E$ are equivalent on any compact subset of $\R^4$.}
\end{re}

Let $0\in W\subset Y\times Z$ such that $W$ is compact and that $\Phi:W\to\Phi(W)$ is a diffeomorphism. Set 
$$\tilde V_f:=\Phi^{-1}(V_f)\subset\R^2\times\{0\}.$$
Then it is easy to check that
\begin{displaymath}\begin{array}{ccl}
$$d_{V_f}(.)|_{\Phi(W)}&=&dist_E(.,V_f)|_{\Phi(W)}$$\\
$$&\sim& dist_E(.,\tilde V_f)|_W$$\\
$$&=&dist_E(.,\{y_3=y_4=0\})|_W$$\\
$$&=:& d_{\tilde V_f}(.)|_W.$$
\end{array}\end{displaymath}

Let $b\in \Phi(W)$ and set $\tilde b=\Phi^{-1}(b)$. Now set $B(\tilde b,R):=\{y\in\R^4:|y_i-\tilde b_i|\le R\}\subset W$, then $B_\Phi(b,R):=\Phi(B(\tilde b,R))$ is a closed neighborhood of $b$.

The following lemma is the key to prove the existence of the limits of the trajectories of $\nabla^hf.$

\begin{lem}\label{localfini} Let $f\in M_d\cap A_d\cap D_d\cap L_d$, we have the following.
\begin{enumerate}
\item[(I)] Let $\alpha:[t_1,t_2]\to B_\Phi(b,R)$ be a horizontal curve such that $\alpha(t_1)=b$ and $\alpha(t_2)\in\partial B_\Phi(b,R).$ Then $l_\delta(\alpha(t))>C_{b,R}>0$ where $C_{b,R}$ does not depend on $\alpha.$

\item[(II)] For each $\dis p\in \partial B_\Phi(a,\frac{R}{2})$ and for any horizontal curve $\dis\alpha:[t_1,t_2]\to B_\Phi(a,R)-\mathring{B}_\Phi(a,\frac{R}{2})$ such that $\alpha(t_1)=p$ and $\alpha(t_2)\in\partial B_\Phi(a,R),$ we have $l_\delta(\alpha(t))>C'_{a,R}>0$ where $C'_{a,R}$ does not depend on $\alpha.$
\end{enumerate}
 
\end{lem}
\begin{proof} (I) By compactness, it is sufficient to prove that $l_\delta(\alpha(t))>0$ for such $\alpha(t)$. Set $q:=\alpha(t_2)$ and $\tilde q:=\Phi^{-1}(q)=(\tilde q_1,\tilde q_2,\tilde q_3,\tilde q_4)$. We have four cases to consider:
\begin{enumerate}
\item[(i)] $b\not\in V_f$;
\item[(ii)] $b\in V_f$ and $|\tilde q_3-\tilde b_3|=R$ or $|\tilde q_4-\tilde b_4|=R$;
\item[(iii)] $b\in V_f$, $|\tilde q_3-\tilde b_3|<R$ and $|\tilde q_4-\tilde b_4|<R$.
\end{enumerate}

\begin{cla}\label{rie} We have $l_g(\alpha(t))\ge cR$ where $c$ is a positive constant.
\end{cla}
\begin{proof} Since $\Phi$ is a diffeomorphism, it is bi-Lipschitz, so by Remark \ref{equiv}, we get
$$l_g(\alpha(t))\ge dist(b,q)\sim dist_E(b,q)\sim dist_E(\tilde b,\tilde q)\ge R.$$
\end{proof}

We need first the following claim whose proof follows easily from Claim \ref{rie} and the formulas (\ref{lg}), (\ref{ldelta}).
\begin{cla}\label{far} If $d_{V_f}(\alpha(t))\ge r>0$, then $l_\delta(\alpha(t))\ge r l_g(\alpha(t))\ge crR$.
\end{cla}

Now consider case (i). Set $2r:=d_{V_f}(b)>0$. If $d_{V_f}(\alpha(t))> r$ for all $t\in [t_1,t_2]$, then by Claim \ref{far}, it follows that $$l_\delta(\alpha(t))\ge crR>0.$$
 Otherwise, there exists $t_0\in(t_1,t_2]$ such that $d_{V_f}(\alpha(t_0))=r$ and $d_{V_f}(\alpha(t))> r$ for $t\in[t_1,t_0]$. Then
\begin{displaymath}\begin{array}{ccl}
$$l_\delta(\alpha(t))&\ge& l_\delta(\alpha(t)|_{[t_1,t_0]})$$\\
$$&\ge& rl_g(\alpha(t)|_{[t_1,t_0]})$$\\
$$&\ge& r dist(b,\alpha(t_0))$$\\
$$&\sim& r dist_E(b,\alpha(t_0))$$\\
$$&\ge& r^2>0$$
\end{array}\end{displaymath}

Next, for case (ii), by the assumptions, we have 
$$dist_E(\tilde q,\{y_3=y_4=0\}\ge R.$$
Hence $$R\le dist_E(\tilde q,\tilde V_f)\sim dist_E(q,V_f)=d_{V_f}(q)$$
which means that there is $r'>0$ such that $d_{V_f}(q)\ge 2r'$. By applying the arguments as in the case (i), it follows that $l_\delta(\alpha(t))>0.$

Let us consider the case (iii), since $q\in\partial B_\Phi(b,R)$, we have $\tilde q\in\partial B(\tilde b,R)$. So $|\tilde q_1-\tilde b_1|=R$ or $|\tilde q_2-\tilde b_2|=R$. Let $\tilde\alpha(t)=\Phi^{-1}(\alpha(t))$. For $y\in W$, set 
$$\tilde\Delta_y=d\Phi^{-1}(\Delta_{\Phi(y)}),$$ which defines a distribution $\tilde\Delta$ on $W$. Let $y\in W$ and $X,Y\in T_y\tilde\Delta_y$, the "pullbacks" $\tilde g_{SR}$ of $g_{SR}$ and $\tilde\delta_{SR}$ of $\delta_{SR}$ on $\tilde\Delta$ at $y$ are given respectively by
$$\tilde g_{SR}(X,Y):=g_{SR}(d_y\Phi(X),d_y\Phi(Y)),$$
$$\tilde \delta_{SR}(X,Y):=\delta_{SR}(d_y\Phi(X),d_y\Phi(Y))=d^2_{V_f}(\phi(y)).g_{SR}(d_y\Phi(X),d_y\Phi(Y)).$$
Since all metrics on $\R^2$ are locally equivalent, by compactness, we have $\tilde g_{SR}(X,Y)\sim g_E|_{\tilde\Delta}$.  Moreover $d_{V_f}\sim d_{\tilde V_f}$. So 
$$\tilde \delta_{SR}(X,Y)\sim d^2_{\tilde V_f}.g_E|_{\tilde\Delta}.$$
Denote by $l_{\tilde\delta}(\tilde\alpha(t))$ the length of $\tilde\alpha(t)$ with respect to the metric $\tilde\delta_{SR}$. Then by this observation, it is sufficient to prove that $l_{\tilde\delta}(\tilde\alpha(t))>0.$

Without loss of generality, we may assume that $\tilde\alpha(t)$ is parametrized by its arc length with respect to the metric $g_E$, i.e., $||\dot{\tilde\alpha}(t)||_E=1$ for $t\in[t_0,t_2]$, where $||.||_E$ denotes the Euclidean norm. Let $\xi(t)$ be the orthogonal projection of $\tilde\alpha(t)$ on $\tilde V_f$ with respect to the Euclidean metric $g_E$, so $\dot\xi(t)$ is the projection of $\dot{\tilde\alpha}(t)$ on $T_{\xi(t)}\tilde V_f=\R^2\times\{0\}$. Note that $dist_E(\xi(t_1),\xi(t_2))=\dist_E(\tilde b,(\tilde q_1,\tilde q_2,0,0))\ge R$. Denote by $\G(2,4)$ the Grassmannian space of $2-$planes in $\R^4$, the distance between two planes $P,Q\in\G(2,4)$ is given by
$$\dis\rho(P,Q):=\max_{p\in P,||p||=1}\min_{q\in Q}||p-q||_E.$$
Since the map $\tilde\Delta:\R^4\to\G(2,4), y\mapsto \tilde\Delta_y$ is smooth, it is locally Lipschitz, so we may suppose that for all $x,y\in W$, we have
$$\dis\rho(\tilde\Delta_x,\tilde\Delta_y)\le K||x-y||_E,$$
where $K$ is a positive constant. Hence for $t\in[t_0,t_2]$,
$$\dis\rho(\tilde\Delta_{\alpha(t)},\tilde\Delta_{\xi(t)})\le K||\alpha(t)-\xi(t)||_E.$$
Let $\dot\eta(t)$ be the orthogonal projection of $\dot{\tilde\alpha}(t)$ on $\tilde\Delta_{\xi(t)}\subset\{0\}\times\R^2$ with respect to the Euclidean metric $g_E$. Since $\tilde\Delta_{\xi(t)}\perp T_{\xi(t)}\tilde V_f$, we have $\dot{\tilde\alpha}(t)=\dot\xi(t)+\dot\eta(t)$. So
$$\dis||\dot\xi(t)||_E=||\dot{\tilde\alpha}(t)-\dot\eta(t)||_E\le\rho(\tilde\Delta_{\tilde\alpha(t)},\tilde\Delta_{\xi(t)})\le K||\tilde\alpha(t)-\xi(t)||_E\le Kd_{\tilde V_f}(\tilde\alpha(t)).$$
Consequently
\begin{displaymath}\begin{array}{ccl}
$$\l_\delta(\tilde\alpha)&=&\dis\int_{t_0}^{t_2}|\dot{\tilde\alpha}(t)|dt$$\\
$$&=&\dis\int_{t_0}^{t_2}d_{\tilde V_f}(\tilde\alpha(t))dt$$\\
$$&\ge&\dis\frac{1}{K}\int_{t_0}^{t_2}||\dot\xi(t)||_Edt$$\\
$$&\ge&\dis\frac{1}{K}dist_E(\xi(t_1),\xi(t_2))\ge R.$$
\end{array}\end{displaymath}
This end the proof of (I). It is clear that (II) follows easily from (I). The lemma follows.
\end{proof}

Now we have all the ingredients necessary to prove Theorem \ref{havelimit}.
\begin{proof}[Proof of Theorem \ref{havelimit}]
Let $f\in M_d\cap A_d\cap D_d\cap L_d$ and let $\alpha(t)\subset B$ be a trajectory of $f$ in $B$, assume that $\alpha(t)$ is parametrized by the values of $f$. If $\alpha(t)$ does not approach to $V_f$, then by the same arguments as Proposition 5.4, it follows that $\alpha(t)$ has a limit on $\partial B.$ So from now on, assume that $\alpha(t)$ approaches to $V_f$, i.e., there exists a sequence $t_m\to t_0<+\infty$ such that 
$$d_{V_f}(\alpha(t_m))\to 0.$$
Since $\alpha(t)$ is parametrized by the values of $f$, we have 
$$dist_E(\alpha(t),f^{-1}(t_0)\cap B)\to 0 \text{ when }t\to t_0,$$
where $dist_E$ denotes the distance function to $V_f$ with respect to the usual Euclidean metric on $\R^4$. Thus
\begin{equation}\label{tendtolvsf}dist_E(\alpha(t),V_f\cap f^{-1}(t_0)\cap B)\to 0 \text{ when }t\to t_0.\end{equation}
We need to prove that $\alpha(t)$ has only one accumulation point in $V_f\cap f^{-1}(t_0)\cap B$. By contradiction, assume that $\alpha(t)$ has two accumulation points $a,a'$ in $V_f\cap f^{-1}(t_0)\cap B$.

If $a\not\in\Gamma_f$, then there are a local diffeomorphism $\Phi$ and a "box" $B_{\Phi}(a,R)$ satisfying the conclusions such that Lemma \ref{localfini} holds. Then it is clear that for $R$ small enough, $\alpha(t)\cap\left(B_\Phi(a,R)-\mathring{B}_\Phi(a,\frac{R}{2})\right)$ contains infinity many components with one end point in $\partial B_\Phi(a,R)$ and another in $\dis \partial B_\Phi(a,\dis\frac{R}{2})$. According to Lemma \ref{localfini}, the length of these components with respect to the metric $\delta_{SR}$ is bounded from below by a positive constant, which implies that $l_\delta(\alpha(t))=+\infty$. This contradicts Lemma \ref{finitelength}.

If $a'\in \Gamma_f$, by the same arguments, we get a contradiction. So suppose that $a,a'\in \Gamma_f$. By Theorem \ref{noncontient} and Lemma \ref{conhori}, the intersection $f^{-1}(t_0)\cap \Gamma_f$ is finite, so let $R$ be small enough such that $B_{\Phi}(a,R)\cap f^{-1}(t_0)\cap \Gamma_f=\{a\}$. Since $\alpha(t)$ accumulates to $a$, it intersects $\dis \partial B_{\Phi}(a,R)$ infinitely many times. By compactness, the set $\dis \partial B_{\Phi}(a,R)\cap\alpha(t)$ has an accumulation point, denoted by $a_1$. By (\ref{tendtolvsf}), it is clear that $a_1\in f^{-1}(t_0)$. Hence $a_1\not\in\Gamma_f$. By the same arguments as above, we get again a contradiction. This ends the proof of Theorem \ref{havelimit}.
\end{proof}

%ACKNOWLEDGMENTS. This work was supported by the Fields Institute and NAFOSTED (Vietnam) grant 101.01-2011.44.


\begin{thebibliography}{99}
\addcontentsline{toc}{chapter}{Bibliographie}
\bibitem{AS}{A. A. Ardentov \&  Yu. L. Sachkov,} {\it Conjugate points in nilpotent sub-Riemannian problem on the Engel group}, 
Journal of Mathematical Sciences
{\bf 195} (2013) no. 3,  369-390.
\bibitem{BHT}{Z. M. Balogh, I. Holopainen \& J. T. Tyson, {\it Singular solutions, homogeneous norms, and quasiconformal mappings in Carnot groups}, Math. Ann. {\bf 324} (2002), 159-186.}
\bibitem{Bel}{A. Bella\"iche, {\it The tangent space in sub-riemannian geometry}, Progress in Mathematics Vol.144, Birkh\"auser Verlag, 1996.}
\bibitem{BR}{R. Benedetti \& J-J. Risler, {\it Real algebraic and semi-algebraic sets}, Hermann, 1991.}
\bibitem{BCR}{J. Bochnak, M. Coste \& M-F. Roy,  {\it G\'eom\'etrie alg\'ebrique r\'eelle}, Springer, 1987.}
\bibitem{Cho}{W.L. Chow, {\it \"Uber Systeme von linearen partiellen Differentialgleichungen erster Ordnung}, Math. Ann. {\bf 117} (1939), 98-105.}
\bibitem{Cos}{M. Coste,  {\it An Introduction to semialgebraic geometry}, Dip. Mat. Univ. Pisa, Dottorato di Ricerca in Matematica, Istituti Editoriali e Poligrafici Internazionali, Pisa, 2000. http://perso.univ-rennes1.fr/michel.coste/articles.html.}
\bibitem{DKO}{S. T. Dinh, K. Kurdyka \& P. Orro, {\it Gradient horizontal de fonctions polynomiales}, Annales de l'Institut Fourier {\bf 59} (2009) no. 5, 1999-2042.}
\bibitem{Ehr}{C. Ehresmann, {\it Les connexions infinit\'esimales dans un espace fibr\'e diff\'erentiable}, Colloque de Topologie, Bruxelles (1950), 29-55.}
\bibitem{Eng}{F. Engel, {\it Zur Invariantentheorie der Systeme Pfaff'scher Gleichungen}, Leipz. Ber. Band {\bf 41} (1889), 157-176.}
%\bibitem{Gab}{A. Gabrielov: {\it Multiplicities of zeroes of polynomials on trajectories of polynomial vector fields and bounds on degree of nonholonomy}, Math. Res. Lett. {\bf 2} (1995),  no. 4, 437-451.}
\bibitem{Ger}{V. Gershkovich, {\it Exotic Engel structures on $\R^4$},  Russian J. Math. Phys. {\bf 3} (1995),  no. 2, 207-226.}
\bibitem{GM}{M. Goresky \& R.MacPherson,  {\it Stratified Morse theory}, Springer, 1988.}
\bibitem{Gro}{M. Gromov: {\it Carnot-Caratheodory spaces seen from within}, IHES, 1994.}
\bibitem{GP}{V. Guillemin \& A. Pollack, {\it Differential topology}, Prentice-Hall, 1974.}
\bibitem{Har}{R. M. Hardt: {\it Semi-algebraic local-triviality in semi-algebraic mappings}, Amer. J. Math. 102 (1980), no. 2, 291-302.}
\bibitem{Hir}{M. Hirsch: {\it Differential topology}, Springer, 1976.}
\bibitem{Jea}{F. Jean: {\it Entropy and complexity of a path in sub-riemannian geometry}, ENSTA, 1999.}
\bibitem{KMS}{I. Kol\'ar, P. W. Michor \& J. Slov\'ak, {\it Natural operations in differential geometry}. Springer-Verlag, Berlin, 1993.}
\bibitem{KOS}{K. Kurdyka, P. Orro \& S. Simon, {\it Semialgebraic Sard theorem for generalized critical values}, J. Diff. l Geom.  {\bf 56}  (2000),  no. 1, 67-92.}
\bibitem{Loj2}{S. {\L}ojasiewicz, {\it Ensembles semi-analytiques}, I.H.E.S, Bures-sur-Yvette, 1965.}
\bibitem{Loj3}{S. {\L}ojasiewicz, {\it Sur la g\'eom\'etrie semi- et sous- analytique}, Annales de l'institut Fourier. {\bf 43} (1993) no 5, 1575-1595.}
\bibitem{Loj} S. \L ojasiewicz, {\it Sur le probl\`eme de la division}, Studia Math. {\bf 18}, (1959), 87-136.
\bibitem{Loj4}{S. {\L}ojasiewicz, {\it Sur les trajectoires du gradient d'une fonction analytique},  Seminari di Geometria (Bologna), (1982-1983),  115-117.}
\bibitem{Mag}{V. Magnani, {\it A Blow-up theorem for regular hypersurfaces on nilpotent groups}, Manuscripta Math. {\bf 110} (2003), no. 1, 55-76.} 
\bibitem{Mil}{J. Milnor, {\it On the Betti numbers of real varieties}. Proc. Amer. Math. Soc. {\bf 15}, (1964), 275-280.} 
%\bibitem{Or}{P. Orro: {\it Quelques propri\'et\'es de la distance g\'eod\'esique}, Real analytic and algebraic singularities (Nagoya/Sapporo/Hachioji, 1996), 107-113, Pitman Res, Notes Math. Ser., 381, Longman, Harlow, 1998.}
\bibitem{Rab}{P.J. Rabier, {\it Ehresmann fibrations and Palais-Smale conditions for morphisms of Finsler manifolds}. Ann. of Math. (2) {\bf 146} (1997), no. 3, 647-691.} 
%\bibitem{Ras}{P.K. Rashevsky: {\it Any two points of a totally nonholonomic space may be connected by an admissible line}, Uch. Zap. Ped. Inst. im. Liebknechta. Ser. Phys. Math. {\bf 2} (1938), 83-94.}
\bibitem{Stri}{R.S. Strichartz, {\it Sub-riemannian geometry}, J. Diff. Geom. {\bf 24} (1986). 221-263.}
\bibitem{Sus}{H.J. Sussmann, {\it Orbits of families of vector fields and integrability of distributions}, Trans. Amer. Math. Soc, {\bf 180} (1973), 171-188.}
\bibitem{Tho}{R. Thom, {\it Sur l'homologie des vari\'et\'es alg\'ebriques r\'eelles}. Differential and Combinatorial Topology (A Symposium in Honor of Marston Morse). Princeton Univ. Press, Princeton, N.J, (1965), 255-265}
\bibitem{Vog}{T. Vogel, {\it Existence of Engel structures}, Ann. of Math. (2) {\bf 169} (2009), no. 1, 79-137.}
\bibitem{Zhi}{M. Ya. Zhitomirskii, {\it Normal forms of germs of two-dimensional distributions on $\R^4$}, (Russian)  Funktsional. Anal. i Prilozhen, {\bf 24} (1990),  no. 2, 81-82;  translation in  Funct. Anal. Appl.  {\bf 24} (1990),  no. 2, 150-152.}

\end{thebibliography}
\end{document}